\documentclass[11pt, oneside]{amsart}
\usepackage[utf8]{inputenc}

\usepackage{amsmath}
\usepackage{amsthm}
\usepackage{amssymb} 
\usepackage[mathscr]{euscript}
\usepackage{amsmath}
\usepackage{amsthm}
\usepackage{amssymb}
\usepackage{fancyhdr}
\usepackage{enumerate}
\usepackage{amscd}
\usepackage[all]{xy}
\usepackage{graphicx}
\usepackage{color}
\usepackage[dvipsnames]{xcolor}
\usepackage{hyperref}
\hypersetup{
    colorlinks,
    citecolor=black,
    filecolor=black,
    linkcolor=black,
    urlcolor=black
}
\usepackage{bm}
\usepackage{tikz}
\usepackage[all]{xy}
\usepackage{graphicx}
\usetikzlibrary{backgrounds}
\usepackage{soul}

\newcommand{\margin}[1]{\marginpar{\parbox{0in}{\parbox{0.8in}{\color{blue} \raggedright \scriptsize #1}}}}

\theoremstyle{plain}
\newtheorem{lemma}{Lemma}[section]

\newtheorem{thm}[lemma]{Theorem}
\newtheorem{cor}[lemma]{Corollary}

\numberwithin{equation}{section}

\theoremstyle{remark}

\theoremstyle{definition}
\newtheorem{rmk}[lemma]{Remark}

\newtheorem{defin}[lemma]{Definition}

\DeclareMathOperator{\Div}{div}
\DeclareMathOperator{\tr}{tr}
\DeclareMathOperator{\Int}{Int}

\newcommand{\too}{\longrightarrow}
\newcommand{\rr}{\mathbb{R}}

\newtheorem{manualtheoreminner}{Theorem}
\newenvironment{manualtheorem}[1]{%
  \renewcommand\themanualtheoreminner{#1}%
  \manualtheoreminner
}{\endmanualtheoreminner}

\usepackage[lmargin=1.2in,rmargin=1.2in]{geometry}
\begin{document}
\title[Equality in the spacetime positive mass theorem II]{Equality in the spacetime positive mass theorem II}
\author{Lan-Hsuan Huang}
\address{Department of Mathematics, University of Connecticut, Storrs, CT 06269, USA}
\email{lan-hsuan.huang@uconn.edu}
\author{Dan A. Lee}
\address{CUNY Graduate Center and Queens College}
\email{dan.lee@qc.cuny.edu}
\thanks{The first named author was partially supported by the NSF CAREER Award DMS-1452477 and NSF DMS-2005588. The authors thank Harvard’s Center of Mathematical Sciences and Applications for their hospitality while some of this research was being carried out. 
}

\maketitle

\begin{abstract}
We provide a new proof of the equality case of the spacetime positive mass theorem, which states that if a complete asymptotically flat initial data set $(M, g, k)$ satisfying the dominant energy condition has null ADM energy-momentum (that is, $|E|=|P|$), then $(M,g)$ must isometrically embed into Minkowski space with $k$ as its second fundamental form. Previous proofs either used spinor methods~\cite{Witten:1981, Beig-Chrusciel:1996, Chrusciel-Maerten:2006}, relied on the Jang equation~\cite{Huang-Lee:2020, Eichmair:2013}, or assumed three spatial dimensions~\cite{Hirsch-Zhang}. In contrast, our new proof only requires knowing that $E\ge|P|$ for all complete initial data sets near $(g,k)$ on $M$ satisfying the dominant energy condition.

\end{abstract}

\section{Introduction}

Let us first explicitly state the definition of asymptotic flatness we will be using.
\begin{defin}\label{def:AF}
Let $n\ge3$, let $M$ be an $n$-dimensional connected manifold, possibly with boundary, and let $\frac{n-2}{2}<q<n-2$. We say that an initial data set $(M,g,k)$ is \emph{complete asymptotically flat with decay rate $q$} if the following hold: 
\begin{itemize}
    \item Each noncompact end of $M$ (of which there is at least one) is diffeomorphic to \hbox{$\rr^n\setminus \bar B_1(0)$}, and for some $0<\alpha<1$,
\begin{equation} \label{eqn:AF}
 (g-g_{\mathbb{E}},k)\in C^{2,\alpha}_{-q}\times C^{1,\alpha}_{-q-1},
 \end{equation}
where $g_{\mathbb{E}}$ denotes a background metric that is equal to the standard Euclidean metric on each $\rr^n\setminus \bar B_1(0)$ end, and $C^{2,\alpha}_{-q}$ and $C^{1,\alpha}_{-q-1}$ denote weighted H\"{o}lder spaces.
\item The energy density $\mu$ and current density $J$ defined by
\begin{equation} \label{eqn:constraints}
(\mu, J):=\left(\, \tfrac{1}{2}(R_g -|k|^2_g - (\tr_g k)^2)\, ,\, (\Div_g k)^{\sharp} - \nabla(\tr_g k)   \,\right)
\end{equation}
are integrable over $M$, where $(\Div_g k)^{\sharp}$ denotes the vector field dual to the one-form $\Div_g k$.
\end{itemize}

\end{defin}
Note that asymptotic flatness of $(g,k)$ implies that the ADM energy-momentum $(E,P)$ of each end is well-defined. Our main result characterizes the equality case $E=|P|$ of the spacetime positive mass theorem, but we would like to emphasize that, unlike all previous results characterizing $E=|P|$, 
ours follows from the positive mass inequality $E\ge|P|$ itself and does not rely on any particular proof of this inequality. Consequently, we introduce the following definition.

\begin{defin}
Let $(M, g, k)$ be a complete asymptotically flat initial data set with decay rate $q$, possibly with boundary. We say that \emph{the positive mass theorem is true near $(g, k)$} if there is an open ball around $(g, k)$ in $\left(g_{\mathbb E}+C^{2,\alpha}_{-q}\right)\times C^{1,\alpha}_{-q-1}$, for some $\alpha$, such that for each complete asymptotically flat initial data set $(\bar{g}, \bar{k})$ with decay rate $q$ in that open ball, satisfying the dominant energy condition and having weakly outer trapped boundary (if there is a boundary), we have $\bar E\ge |\bar P|$, where $(\bar E, \bar P)$ denotes the ADM energy-momentum of any asymptotically flat end of  $(M, \bar g, \bar k)$. 
\end{defin} 

We say the boundary $\partial M$ of an asymptotically flat manifold $M$ is  \emph{weakly outer trapped} if its outward null expansion $\theta^+:=  H + \tr_{\partial M} k$ satisfies $\theta^+\le0$, with respect to the normal that points toward $M$.\footnote{Our convention is that the mean curvature $H$ is the tangential divergence of the chosen unit normal vector. For example, the mean curvature of a round sphere in Euclidean space with respect to the outward normal is positive.}

It follows from E. Witten's proof of the positive mass theorem~\cite{Witten:1981} (and~\cite{GHHP:1983, Herzlich:1998} for the boundary case) that the positive mass theorem is true near any $(g, k)$ on a \emph{spin} manifold. Building on the classic works of R. Schoen and S.-T. Yau~\cite{Schoen-Yau:1979-pmt1, Schoen-Yau:1981-pmt2}, in joint work with M. Eichmair and Schoen we proved that the positive mass theorem is true\footnote{The definition of asymptotic flatness used in~\cite{EHLS:2016} is slightly different from the one stated here. 
By going through the proofs there with more care, one sees that Definition~\ref{def:AF} is adequate for proving the positive mass theorem. See~\cite{LLU} for details.} near any $(g, k)$ on a manifold of dimension less than 8~\cite{EHLS:2016}, with the boundary case proved by the second author with M. Lesourd and R. Unger~\cite{LLU}. J.~Lohkamp has also announced a proof in all dimensions~\cite{Lohkamp:2016}. 

Understanding the $E=|P|$ case of the spacetime positive mass theorem breaks down into two separate steps: The first step is to show that $E=|P|=0$, and the second step is to use the fact that $E=0$ to find an embedding of the initial data into Minkowski space. In the first paper in this series~\cite{Huang-Lee:2020},
we built on the work of R. Beig and P. Chru\'{s}ciel~\cite{Beig-Chrusciel:1996} (see also~\cite{Chrusciel-Maerten:2006}) in the spin case to construct an argument that proves the first step without appealing to any particular proof of the $E\ge|P|$ inequality. We state this result  below and note that it was generalized to include the possibility of a boundary in~\cite{LLU}.
\begin{thm}\label{thm:E=P}
Let $(M, g, k)$ be a complete asymptotically flat initial data set with decay rate~$q$, possibly with boundary, and assume that the positive mass theorem is true near $(g,k)$. Furthermore, we make the stronger asymptotic assumption that 
with this decay rate $q$, there exists $0<\alpha<1$ satisfying~\eqref{eqn:AF} such that 
\begin{align}\label{eq:decay}
q+\alpha>n-2,
\end{align}
and 
\begin{equation}
\label{eq:constraint_decay}
(\mu, J)\in C^{2,\alpha}_{-n-\delta},
\end{equation}
for some $\delta>0$, where $n$ is the dimension.  If $(g,k)$ satisfies the dominant energy condition and the boundary (if there is one) is weakly outer trapped, then for each asymptotically flat end, $E=|P|$ implies that $E=|P|=0$. 
\end{thm}
Although the stronger asymptotic assumption \eqref{eq:decay} may seem undesirable in higher dimensions (note that \eqref{eq:decay} can be dropped for $n=3, 4$ if $\alpha$ in \eqref{eqn:AF} is close to $1$), we found pp-wave counterexamples showing that the theorem is false without it for $n>8$. See~\cite[Example 7]{Huang-Lee} for details. 

Once one knows that $E=0$, the ``second step'' described above follows from work of Eichmair~\cite{Eichmair:2013} if $n<8$. More specifically, following Schoen and Yau's proof for the $n=3$ case\footnote{When $n=3$, one must make an extra decay assumption on $\tr_g k$.}~\cite{Schoen-Yau:1981-pmt2}, Eichmair used the Jang equation  to prove the inequality $E\ge0$, and then by analyzing the case of equality in that proof, he showed that the $E=0$ case can only occur when the Jang solution provides a graphical isometric embedding of $(M, g)$ into Minkowski space with $k$ as its second fundamental form. 
The main result of this article is to replace this argument for the ``second step'' with one that naturally extends the methods used to prove Theorem~\ref{thm:E=P} and is self-contained in the sense that it does not depend on how one proves that $E\ge|P|$ (nor on how one proves $E\ge0$).

\begin{thm}\label{thm:main}
Let $(M, g, k)$ be a complete asymptotically flat initial data set with decay rate~$q$, possibly with boundary, and assume that the positive mass theorem is true near $(g,k)$. Further assume that $(g,k)$ is locally $C^5\times C^4$. If $(g,k)$ satisfies the dominant energy condition and the boundary (if there is one) is weakly outer trapped, then $E>0$ in each asymptotically flat end, unless $(M,g)$ has no boundary and isometrically embeds into  Minkowski space with $k$ as its second fundamental form, in which case $E=|P|=0$. 
\end{thm}
Combining the two previous theorems gives  the following result. 
\begin{cor}\label{cor:main}
Let $(M, g, k)$ be a complete asymptotically flat initial data set with decay rate~$q$, possibly with boundary,  and assume that the positive mass theorem is true near $(g,k)$. Further assume that \eqref{eq:decay} and~\eqref{eq:constraint_decay} hold, and that $(g,k)$ is locally $C^5\times C^4$.  If $(g, k)$ satisfies the dominant energy condition and the boundary (if there is one) is weakly outer trapped, then $E>|P|$ in each asymptotically flat end, unless $(M, g)$ has no boundary and embeds into Minkowski space with $k$ as its second fundamental form, in which case $E=|P|=0$. 
\end{cor}

This fact was essentially already proved in dimension 3 by Beig and Chru\'{s}ciel~\cite{Beig-Chrusciel:1996} and for general spin manifolds by Chru\'{s}ciel and D. Maerten~\cite{Chrusciel-Maerten:2006}. These results required a slightly stronger assumption than~\eqref{eq:decay} but do not require $C^5\times C^4$ regularity. More recently, Sven Hirsch and Yiyue Zhang~\cite{Hirsch-Zhang} gave a new proof in dimension $3$ that avoids assuming any of~\eqref{eq:decay}, \eqref{eq:constraint_decay}, or $C^5\times C^4$ regularity.

The basic outline of the proof of Theorem~\ref{thm:main} is the following: As in the proof of Theorem~\ref{thm:E=P}, we use the fact that $(g, k)$ minimizes a modified Regge-Teitelboim Hamiltonian subject to a constraint, and then we invoke Lagrange multipliers to construct lapse-shift pairs that satisfy a Killing initial data type of equation. The difference is that when $E=0$, $(g,k)$ minimizes many different 
Regge-Teitelboim Hamiltonians, so we can actually construct an entire $(n+1)$-dimensional space of these lapse-shift pairs that can be thought of as being asymptotic to the $(n+1)$-dimensional space of translational Killing fields on Minkowski space as we approach spatial infinity. 
 The essential new ingredient is that by invoking our recent work~\cite{Huang-Lee}, 
we can also make sure that these lapse-shift pairs satisfy what we call the ``$J$-null-vector equation.''
 Moreover, having so many solutions implies that $(g,k)$ is vacuum. All of this is explained in Section~\ref{sec:Lagrange}, culminating in the statement of Lemma~\ref{lem:KID}. Observe that proving that $(g,k)$ is vacuum is itself a difficult task.   Note that the vacuum condition $\mu=|J|_g=0$ is a much stronger condition than the ``borderline'' $\mu=|J|_g$ case of the dominant energy condition. The fact that $E=|P|$ implies $\mu=|J|_g$ is a previously unstated direct consequence of~\cite[Theorem 5.2]{Huang-Lee:2020} and~\cite[Theorem 6]{Huang-Lee}, but the aforementioned pp-wave counterexamples from~\cite[Example 7]{Huang-Lee} demonstrate that $E=|P|$ does not imply vacuum.

 Once we know that $(g,k)$ is vacuum, it follows that each of these lapse-shift pairs is actually vacuum Killing initial data for $(g,k)$, and we can extend them to become actual Killing fields on the vacuum Killing development of $(g,k)$. The next step is to show that on an asymptotically flat Lorentzian manifold, having such a space of Killing fields that are asymptotic to the $(n+1)$ translational directions implies that the Lorentzian manifold must be Minkowski space. This observation is described in Theorem~\ref{thm:Killing0} below, which requires no curvature assumptions.

\begin{thm}\label{thm:Killing0} 
Let $(\mathbf N, \mathbf g)$ be an $(n+1)$-dimensional connected $C^2$ Lorentzian spacetime, possibly with boundary, that is asymptotically flat in the sense that a subset of $\mathbf N$ is diffeomorphic to $\rr\times (\rr^n\setminus \bar B_1(0))$, and in these $x^0,\ldots,x^n$ coordinates, 
\[
	\mathbf g_{\mu\nu} (x^0, x)= \eta_{\mu\nu} + O_2(|x|^{-\beta})
\]
for some $\beta>0$, where $\eta_{\mu\nu}$ is the standard Minkwoski metric and $|x|$ denotes the norm of the spatial coordinates $x:=(x^1,\ldots,x^n)$ only. More specifically, the $O_2$ notation means that there exists a positive number $C$ such that 
\[
\sum_{|I|=0}^2 \big| |x|^{\beta+|I|} \partial^I (\mathbf g_{\mu\nu} - \eta_{\mu\nu})(x^0, x)\big|  \le C,
\]
where $\partial^I $ denotes the partial derivatives with respect to all $(n+1)$ coordinates.

 Assume that $(\mathbf N, \mathbf g)$ is equipped with Killing fields $\mathbf Y_0,\ldots, \mathbf Y_n$ that are asymptotic to the coordinate translation directions in the sense that 
 \[ \mathbf{Y}_\mu=\frac{\partial}{\partial x^\mu} + O_2(|x|^{-\beta}). \]
Then $\mathbf Y_0,\ldots, \mathbf Y_n$ form a covariant constant global Lorentzian orthonormal frame, and in particular, $\mathbf g$ must be flat.
\end{thm}

As a direct consequence, the above theorem also implies a version for asymptotically flat \emph{Riemannian} manifolds, which is included in Corollary~\ref{cor:KillingRiemannian} below. 

Note that we do not assume $(\mathbf N, \mathbf g)$ is complete in Theorem~\ref{thm:Killing0}. If $(\mathbf N, \mathbf g)$ is complete without boundary, then it must be Minkowski space by the Killing-Hopf theorem. To rule out the possibility of having a weakly outer trapped boundary, we prove the following general result.

\begin{thm}\label{thm:no-MOTS-stationary}
Let $\Omega$ be a compact connected manifold with boundary such that $\partial \Omega=\Sigma_{\mathrm{in}}\sqcup \Sigma_{\mathrm{out}}$ where $\Sigma_{\mathrm{in}}$ and $\Sigma_{\mathrm{out}}$ are both closed and nonempty. Let $\mathbf{N} = [0,\infty)\times \Omega$, and let $\mathbf g$ be a $C^2$  Lorentzian metric on $\mathbf{N}$ such that $\{ 0\}\times \Omega$ is spacelike. Suppose the following holds:
\begin{enumerate}
\item $\frac{\partial}{\partial t}$ is a global timelike Killing vector field, where $t$ denotes the coordinate for the $[0,\infty)$ factor. Thus $(\bf N, \bf g)$ is strictly stationary.
\item $\mathbf g$ satisfies the null energy condition.
\item  $\{0\}\times \Sigma_{\mathrm{out}}$ is a strictly outer untrapped surface (that is, $\theta^+>0$) with respect to the ``outward'' future null normal pointing out of $\mathbf{N}$.
\end{enumerate}
Then $\{0\}\times \Sigma_{\mathrm{in}}$ cannot be a weakly outer trapped surface (that is, it cannot have $\theta^+\le0$ on all of $\{0\}\times \Sigma_{\mathrm{in}}$) with respect to the ``outward'' future null normal pointing toward the interior of $\mathbf N$.
\end{thm}

 When the spatial dimension is less than $8$, Theorem~\ref{thm:no-MOTS-stationary} follows from an elegant argument of A.~Carrasco and M.~Mars (used to prove~\cite[Corollary~1]{Carrasco-Mars:2008}). The dimension restriction is needed because their proof relies on the existence theory of smooth MOTS by Eichmair~\cite{Eichmair:2009}. Our proof is valid in all dimensions and also applies to  more general situations (e.g. the stationary assumption can be significantly relaxed). See Theorem~\ref{thm:no-MOTS-alt}. 

For the special case of static spacetime metrics, which are of the form $\mathbf g = -f^2 dt^2 + g$ where  $f>0$, the null energy condition becomes the assumption 
\begin{align}\label{eq:null}
	f\mathrm{Ric}_{g} - \nabla^2 f + (\Delta_g f) g \ge 0.
\end{align}

\begin{cor}
Let $(\Omega, g)$ be a compact connected $C^2$ Riemannian manifold with boundary such that $\partial \Omega= \Sigma_{\mathrm{in}}\sqcup \Sigma_{\mathrm{out}}$ where $\Sigma_{\mathrm{in}}$ and $\Sigma_{\mathrm{out}}$ are both closed and nonempty. Suppose there is a $C^2$ function  $f>0$ on $\Omega$ such that~\eqref{eq:null} holds.

Suppose $\Sigma_{\mathrm{out}}$ has mean curvature $H>0$ with respect to the unit normal pointing out of~$\Omega$. Then $\Sigma_{\mathrm{in}}$ cannot have $H\le 0$ everywhere, with respect to the unit normal pointing into $\Omega$.
\end{cor}

Specializing even more to the case where $f=1$ reduces to a well-known classical theorem for manifolds with nonnegative Ricci curvature, stated as Theorem~\ref{thm:Frankel}.  Note that many examples, such as static vacuum metrics and electro-vacuum metrics, with any cosmological constant, satisfy condition~\eqref{eq:null}.

The paper is organized as follows. In Section~\ref{sec:Lagrange}, we explore the Lagrange multiplier method carried out in \cite{Huang-Lee:2020} for the $E=0$ case using the new results from \cite{Huang-Lee} to construct an $(n + 1)$-dimensional space of asymptotically translational vacuum Killing initial data. In Section~\ref{sec:Killing}, we prove Theorem~\ref{thm:Killing0} and then complete the proof of the main result, Theorem~\ref{thm:main}, for the case of no boundary.  Finally, the possibility of having a nonempty weakly outer trapped boundary is ruled out by Theorem~\ref{thm:no-MOTS-stationary}, whose proof is presented in Section~\ref{sec:boundary} and can be read independently from all prior sections.

\section{Lagrange multipliers}\label{sec:Lagrange}

We follow closely the definitions and notations in \cite{Huang-Lee:2020}. In particular, we refer the reader to \cite[Section 2]{Huang-Lee:2020} for the definitions of the weighted H\"older space $C^{k,\alpha}_{-q}(M)$ and the weighted Sobolev space $W^{k,p}_{-q}(M)$. Assume $p>n$, and let  $\mathscr{M}^{2,p}_{-q}$ denote the set of symmetric $(0,2)$-tensors $\gamma$ such that  $\gamma-g_{\mathbb{E}}\in W^{2,p}_{-q}(M)$ and $\gamma$ is positive definite at each point. By the inclusion $C^{k,\alpha}_{-q}\subset W^{k,p}_{-q'}$, for any $q'<q$, whenever we work on the weighted Sobolev spaces, we shall assume the decay rate $q$ in the weighted Sobolev spaces  to be slightly smaller than the assumed asymptotic decay rate appearing in Definition~\ref{def:AF}.

Initial data $(g,k)$ can  be equivalently described by a pair $(g,\pi)$ where $\pi$ is a symmetric $(2,0)$-tensor called the \emph{conjugate momentum tensor}, which is related to the 
$(0,2)$-tensor $k$ via the equation
\begin{align}\label{eqn:k}
 k_{ij}:=g_{i\ell}g_{jm} \pi^{\ell m} -\tfrac{1}{n-1}(\tr_g \pi)g_{ij}.
 \end{align}
By slight abuse of vocabulary, we will refer to $(g, \pi)$ as initial data. 
We define the \emph{constraint map} $\Phi$ on initial data by
\[ \Phi(g, \pi) = (2\mu, J),
\]
 where $\mu$ and $J$ are functions of $(g,\pi)$, defined by~\eqref{eqn:constraints} combined with~\eqref{eqn:k}. Recall that the \emph{dominant energy condition (or DEC)} is the condition $\mu\ge|J|_g$.

As in~\cite{Huang-Lee}, given an initial data set $(M, g, \pi)$ and a
smooth function $\varphi$ on $M$, we introduce the \emph{$\varphi$-modified constraint operator}, which is 
defined on initial data $(\gamma,\tau)$ on $M$ by
\begin{align}
		\Phi^{\varphi}_{(g,\pi)}  (\gamma, \tau) 
		&:=
		\Phi(\gamma, \tau) + \left(2\varphi |J|_\gamma^2, (\tfrac{1}{2}+\varphi |J|_g) \gamma \cdot J\right),	\label{eqn:modified}	
\end{align}
where $J$ denotes the current density of $(g,\pi)$, and $(\gamma\cdot J)^i := g^{ij} \gamma_{jk} J^k$ denotes the contraction of $\gamma$ and $J$ with respect to $g$. The crucial property of $\Phi^{\varphi}_{(g,\pi)}$ is that for any $(\gamma, \tau)$, if $(g, \pi)$ satisfies the DEC, $|\gamma-g|_g<1$, and 
\[\Phi^{\varphi}_{(g,\pi)} (\gamma, \tau) =\Phi^{\varphi}_{(g,\pi)}(g,\pi),
\]
then $(\gamma, \tau)$ also satisfies the DEC (see Lemma 3.4 of~\cite{Huang-Lee}). Note that this is a property that the ordinary constraint operator $\Phi$ does not have, but  since $\Phi^{\varphi}_{(g,\pi)}$ only differs from $\Phi$ by an affine function of $(\gamma, \tau)$, it behaves similarly from a PDE perspective.

Throughout this article, a \emph{lapse-shift pair} $(f, X)$ on $M$ will simply refer to a function $f$ and a vector field $X$ on $M$.
Assuming that $(g,\pi)$ is locally $C^5\times C^4$, Theorem 4.1 of~\cite{Huang-Lee}, which was the main technical result of~\cite{Huang-Lee}, tells us\footnote{Take $Z=J$ in the statement of that theorem.} that there exists a dense subset $\mathcal D \subset C^3(\Int M)$ such that if $\varphi \in \mathcal D$, then any lapse-shift pair $(f, X)$ solving  
\begin{equation}\label{eqn:single}
\left(\left.D\Phi_{(g, \pi)}^{\varphi}\right|_{(g,\pi)}\right)^*(f, X)=0
\end{equation}
in the interior of $M$ must also satisfy the pair of equations
\begin{align}\label{eqn:pair}
\begin{split}
\left(\left.D\Phi_{(g, \pi)}^{0}\right|_{(g,\pi)}\right)^*(f, X)&=0\\
2fJ+|J|_g X&=0,
\end{split}
\end{align}
where the second equation is what we call the \emph{$J$-null-vector equation}.
Choosing any bounded $\varphi\in \mathcal D$ ensures that the terms in \eqref{eqn:modified} involving $\varphi$ have the appropriate fall-off rates so that
\[
\Phi^{\varphi}_{(g,\pi)}:  \mathscr{M}^{2,p}_{-q}\times W^{1,p}_{-1-q}\too L^p_{-2-q}.
\]

We will use the method of~\cite{Huang-Lee:2020} to construct solutions to~\eqref{eqn:single} and hence also~\eqref{eqn:pair}. The following definition simply generalizes Definition 5.1 of~\cite{Huang-Lee:2020} from the case $\varphi=0$ to allow for general $\varphi$.
 \begin{defin}\label{def:functional}
 Let $(M, g, \pi)$ be a complete asymptotically flat manifold, possibly with boundary, and choose $q$ to be slightly smaller than the asymptotic decay rate of $(g,\pi)$.  Fix an end of $M$ and constants $a\in \mathbb{R}$ and $b\in \mathbb{R}^n$. Let $(\bar f, \bar X)$ be a smooth lapse-shift pair on $M$ such that $(\bar f, \bar X)$ is equal to $\left(a,\sum b^i \frac{\partial}{\partial x^i}\right)$ in the specified end and vanishes on all other ends of $M$ (if any) and in a neighborhood of $\partial M$ (if nonempty). Let $\varphi$ be a smooth bounded function on~$M$.

We define the \emph{$\varphi$-modified Regge-Teitelboim Hamiltonian} corresponding to $(g,\pi)$ and $(\bar f,\bar X)$ 
to be the function $\mathcal{H} :\mathscr{M}^{2,p}_{-q}\times W^{1,p}_{-1-q} \too \mathbb{R} $ given by 
\begin{align}\label{eqn:functional}
\mathcal{H} (\gamma, \tau) =(n-1)\omega_{n-1} \left[2 a E(\gamma,\tau) + b \cdot P(\gamma, \tau)\right] - \int_M \Phi^{\varphi}_{(g,\pi)}(\gamma, \tau) \cdot (\bar f, \bar X)\, d\mu_g
\end{align}
where  $E(\gamma,\tau)$ and $P(\gamma, \tau)$ denote the ADM energy-momentum of the specified end, and
the volume measure $d\mu_g$ and the inner product in the integral are both with respect to~$g$.
\end{defin}
As explained in~\cite{Huang-Lee:2020}, 
although the terms in~\eqref{eqn:functional} are not individually well-defined for arbitrary $(\gamma, \tau) \in \mathscr{M}^{2,p}_{-q}\times W^{1,p}_{-1-q}$,
 the functional $\mathcal{H}$  extends to all of $\mathscr{M}^{2,p}_{-q}\times W^{1,p}_{-1-q}$ in a natural way.

In what follows, we will assume that $a$ and $b$ satisfy $2a\ge |b|$.
Continuing to follow~\cite{Huang-Lee:2020} (with input from~\cite{LLU} for the boundary case), we define
\begin{multline*}
	\mathcal{C}_{(g,\pi)} = \big\{ (\gamma, \tau) \in \mathscr{M}^{2,p}_{-q}\times W^{1,p}_{-1-q} \,\big|\, \Phi^{\varphi}_{(g,\pi)}(\gamma, \tau) = \Phi^{\varphi}_{(g,\pi)}(g, \pi)\text{ and $	\partial M$ (if nonempty) has }\\
	\text{the same outward null expansion $\theta^+$ with respect to both }(\gamma,\tau)\text{ and }(g,\pi)
	\big\}.
\end{multline*}
We claim that that if $(g,\pi)$ satisfies the DEC with $\partial M$  weakly outer trapped (if nonempty) and has $E=|P|=0$, then 
$(g,\pi)$ is a local minimizer of $\mathcal{H}$  in $\mathcal{C}_{(g,\pi)}$. To see why, observe that for all nearby $(\gamma, \tau)\in \mathcal{C}_{(g,\pi)}$, $(\gamma, \tau)$ also satisfies the DEC (as mentioned above, due to~\cite[Lemma 3.4]{Huang-Lee}) and has weakly outer trapped boundary, so the assumption that the positive mass theorem is true near $(g,\pi)$ tells us that $E(\gamma, \tau) \ge |P(\gamma, \tau)|$. (Technically, since $(\gamma, \tau)$ only has Sobolev regularity, we must invoke~\cite[Theorem 4.1]{Huang-Lee:2020} and~\cite[Lemma 4.2]{LLU}.) Therefore
\begin{multline*}
\tfrac{1}{(n-1)\omega_{n-1}}\big(\mathcal{H}(\gamma, \tau)-\mathcal{H}(g, \pi)\big)=2a E(\gamma,\tau) +b \cdot P(\gamma, \tau) \\
\ge 2a  E(\gamma,\tau) - 2a |P(\gamma,\tau)| 
	= 2a\big( E(\gamma,\tau) - |P(\gamma, \tau)|\big) \ge 0 ,
\end{multline*}
where we used the assumptions that $2a\ge |b|$ and $(g,\pi)$ has $E=|P|=0$. Hence the claim is true.

Thus $(g,\pi)$ locally minimizes the functional $\mathcal{H}$ subject to the constraint $\mathcal{C}_{(g,\pi)}$. In the finite dimensional setting, as long as the constraint  carves out a smooth submanifold of a linear space, we can use the method of Lagrange multipliers to gain information about the minimizer. 
In our more general Banach space setting, in order to construct Lagrange multipliers, all we need is for $\mathcal{C}_{(g,\pi)}$ to be the level set of a map on $\mathscr{M}^{2,p}_{-q}\times W^{1,p}_{-1-q}$ whose linearization is surjective (see~\cite[Theorem D.1]{Huang-Lee:2020} for details).  For $\varphi=0$, this surjectivity was established in~\cite[Lemma~2.10]{Huang-Lee:2020} for the case of no boundary, and in~\cite[Proposition 3.9]{LLU} when boundary is present. Meanwhile, the introduction of $\varphi$ does not change anything in these arguments. 

Invoking Lagrange multipliers exactly as in~\cite{Huang-Lee:2020} (or~\cite{LLU} for the boundary case) leads to the existence of a lapse-shift pair $(f,X)\in  C^{2,\alpha}_{\mathrm{loc}}(\Int M)$ such that
\begin{align*}
&\left(\left.D\Phi^{\varphi}_{(g,\pi)}\right|_{(g,\pi)}\right)^*(f, X)=0 \quad \mbox{ in }\Int M,\\
&(f, X)= \left(a, \sum b^i\frac{\partial}{\partial x^i}\right) + O_{2+\alpha}(|x|^{-q})\text{ in the specified end where $E=0$, and}\\
& (f, X)= O_{2+\alpha}(|x|^{-q})\text{ in each other end,}
\end{align*}
where $O_{2+\alpha}(|x|^{-q})$ denotes an object that decays in $C^{2,\alpha}_{-q}$.
Although we will not use this fact elsewhere in the paper, we remark that the existence of even one such lapse-shift pair implies that $M$ can only have one asymptotically flat end.\footnote{The same argument also applies to the $E=|P|$ case and leads to the following fact. If $(M, g, \pi)$ is a complete asymptotically flat initial data set satisfying the dominant energy condition with $E=|P|$, and if the positive mass theorem is true near $(g, k)$, then $M$ has only one end. Note that this is consistent with the pp-waves counterexamples having $E=|P|\neq 0$ from~\cite[Example 7]{Huang-Lee}, which all have only one end.} 
This is because~\cite[Theorem 3.3]{Christodoulou-OMurchadha:1981} implies that it is impossible for a nontrivial solution of~\eqref{eqn:single} to asymptotically vanish in an end.

As described above, as long as we select $\varphi$ so that Theorem~4.1 of~\cite{Huang-Lee} applies, this $(f,X)$ also satisfies~\eqref{eqn:pair} on $\Int M$.
Since we can construct such a solution for any $(a,b)\in \rr^{n+1}$ satisfying $2a\ge |b|$, and the space of all $(f, X)$ satisfying~\eqref{eqn:pair} forms a vector space $\mathcal{N}$, it follows that for \emph{any} $(a,b)\in \rr^{n+1}$, we obtain a lapse-shift pair $(f, X)$ which is both asymptotic to $\left(a, \sum b^i\frac{\partial}{\partial x^i}\right)$ in the desired end and also satisfies \eqref{eqn:pair}.  Moreover, the dimension of $\mathcal{N}$ is at least $n+1$.

But Corollary 6.6 of~\cite{Huang-Lee} states that if the dimension of $\mathcal{N}$ is greater than 1, then $(g, \pi)$ must be vacuum initial data. Since $(g, \pi)$ is vacuum, $\Phi^\varphi_{(g,\pi)}=\Phi^0_{(g,\pi)}=\Phi$. The intuition behind Corollary 6.6 of \cite{Huang-Lee} is the following: $(f, X)$ lying in the kernel of the overdetermined Hessian-type operator $\left(\left.D\Phi^{0}_{(g,\pi)}\right|_{(g,\pi)}\right)^*$ already means that $(f, X)$ is determined by its $1$-jet at any point. If there is a point $p$ where $J\ne0$, we can use~\eqref{eqn:pair} to write $\nabla f$ in terms of $f$, and hence the $0$-jet of $f$ at $p$ determines the $1$-jet of $f$ at $p$, which then determines the $1$-jet of $(f, X)$ at $p$ by the $J$-null-vector equation, which determines $(f, X)$ everywhere. For details, see~\cite{Huang-Lee}.

Putting together the arguments described in this section, we have proved the following lemma.
\begin{lemma}\label{lem:KID}
Let $(M, g, \pi)$ satisfy the hypotheses of Theorem~\ref{thm:main}, but suppose that $E=0$. 
Then $(g,\pi)$ is vacuum, and for any $(a, b)\in \rr^{n+1}$, 
there exists a lapse-shift pair $(f,X)\in  C^{2,\alpha}_{\mathrm{loc}}(\Int M)$ such that
\begin{align*}
&\left(\left.D\Phi\right|_{(g,\pi)}\right)^*(f, X)=0 \quad \mbox{ in }\Int M\\
&(f, X)= \left(a, \sum b^i\frac{\partial}{\partial x^i}\right) + O_{2+\alpha}(|x|^{-q}) \text{ in the specified end  where $E=0$}.
\end{align*}
\end{lemma}

\begin{defin}\label{def:KID}
Let $(M, g, \pi)$ be a vacuum initial data set. Then any lapse-shift pair $(f, X)$ satisfying $\left(\left.D\Phi\right|_{(g,\pi)}\right)^*(f, X)=0$
is called  \emph{vacuum Killing initial data (or a vacuum KID)} on $(M, g, \pi)$.
\end{defin}

\section{Killing fields}\label{sec:Killing}

In this section, we prove Theorem~\ref{thm:Killing0}, whose proof is completely unrelated to the material in the previous section. After that, we will explain how to combine Lemma~\ref{lem:KID} with Theorem~\ref{thm:Killing0} to prove our main theorem, Theorem~\ref{thm:main}. 

But first, we recall how the Killing condition for vector fields on a spacetime restricts to a spacelike slice. This computation was made in~\cite{Beig-Chrusciel:1997} (written in different notation).\footnote{We provided details of this computation in lapse-shift coordinates in the proof of~\cite[Corollary B.3]{Huang-Lee}, but this required the assumption that
 $U$ is transverse to $\mathbf Y$. Performing the computation in Gaussian coordinates as in~\cite{Beig-Chrusciel:1997} shows that this assumption is unnecessary.}

\begin{lemma}\label{le:lapse-shift}
Let $(\mathbf N, \mathbf g)$ be a spacetime admitting a Killing vector field $\mathbf Y$. Let $U$ be a smooth spacelike hypersurface with the induced data $(g, \pi)$ and future unit normal $\mathbf n$. Write $\mathbf Y= 2f \mathbf n+ Y$ along $U$. Then $(f, X)$ satisfies the following  system along $U$, for $i, j, \ell=1,\dots, n$,
\begin{align}\label{eq:Killing-lapse-shift}
\begin{split}
	\tfrac{1}{2} (L_X g)_{ij} &= \left(\tfrac{2}{n-1}(\tr_g \pi) g_{ij} - 2\pi_{ij}\right) f\\
	- \nabla^2_{ij} f + (\Delta_g f)g_{ij} &= \tfrac{1}{2} (L_X \pi)_{ij} + \left( \mathbf G_{ij} -G_{ij}\right) f \\
 &\quad - \left[-\tfrac{3}{n-1}(\tr_g \pi)\pi_{ij} + 2\pi_{i\ell}\pi^{\ell j} \right] f\\
	&\quad -\left[\left( \tfrac{1}{2(n-1)} (\tr_g \pi)^2 - \tfrac{1}{2} |\pi|_g^2 \right) g_{ij}\right] f
\end{split}
\end{align}
where $\mathbf G_{ij}$ denotes the Einstein tensor of $\mathbf g$ restricted on the tangent bundle of $U$, and $G_{ij}$ is the Einstein tensor of~$g$. Or equivalently, in terms of the constraint operator,
\begin{align}\label{eq:Killing-lapse-shift-constraints}
\left(\left.D\Phi\right|_{(g,\pi)}\right)^*(f, X) &= \Big( -\left(\mathbf G_{ij}+\mu g_{ij}\right) f - 
\tfrac{1}{2}g_{ij}\langle J, X\rangle_g  -(J\odot X)_{ij}, 0\Big)
\end{align}
where $(J\odot X)_{ij} := \tfrac{1}{2} (J_i X_j+J_j X_i)$.
\end{lemma}

We recall the statement of Theorem~\ref{thm:Killing0}.
\begin{manualtheorem}{\ref{thm:Killing0}}
Let $(\mathbf N, \mathbf g)$ be an $(n+1)$-dimensional connected $C^2$ Lorentzian spacetime, possibly with boundary, such that a subset of $\mathbf N$ is diffeomorphic to $\rr\times (\rr^n\setminus \bar B_1(0))$, and in these $x^0,\ldots,x^n$ coordinates, 
\[
	\mathbf g_{\mu\nu} (x^0, x)= \eta_{\mu\nu} + O_2(|x|^{-\beta})
\]
for some $\beta>0$. Assume that $(\mathbf N, \mathbf g)$ is equipped with Killing fields $\mathbf Y_0,\ldots, \mathbf Y_n$  such that
 \[ \mathbf{Y}_\mu=\frac{\partial}{\partial x^\mu} + O_2(|x|^{-\beta}). \]
Then $\mathbf Y_0,\ldots, \mathbf Y_n$ form a covariant constant global Lorentzian orthonormal frame, and in particular, $\mathbf g$ must be flat.
\end{manualtheorem}

\begin{proof}


We first claim that for any $\mu, \nu =0, \ldots, n$, $[\mathbf Y_{\mu},\mathbf Y_{\nu} ] =0$. To see why, let $\mathbf Y =[\mathbf Y_{\mu},\mathbf Y_{\nu} ]$. Since Killing fields form a Lie algebra under Lie bracket,  $\mathbf Y$ is also a Killing field.   Let $U\subset \mathbf N$ be any constant $x^0$ slice of the asymptotically flat end of $\mathbf N$. If $\mathbf n$ is the future unit normal to $U$, then we can decompose $\mathbf Y=2f\mathbf n+X$ and  $\mathbf Y_\mu=2f_\mu\mathbf n+X_\mu$ along $U$, and routine computation shows that
\begin{align*}
    f &= X_\mu (f_\nu) - X_\nu (f_\mu) \\
    X &= [X_\mu, X_\nu] +4(f_\mu \nabla f_\nu- f_\nu \nabla f_\mu).
\end{align*}
By our asymptotic assumptions, it follows that $(f, X)= O_1(|x|^{-\beta-1})$. By Lemma~\ref{le:lapse-shift},   $(f, X)$ satisfies the system~\eqref{eq:Killing-lapse-shift}, which in turn implies that $(f, X)$ satisfies a homogeneous system of Hessian-type equations.  By~\cite[Theorem 3.3]{Christodoulou-OMurchadha:1981}, any $(f, X)$ satisfying such a system of equations with decay in  $O_1(|x|^{-\beta-1})$ must vanish identically. 


At any point $p\in\mathbf N$, define $K(p)$ to be the maximum absolute value of all sectional curvatures of $\mathbf g$ at $p$. Thus $K(p)=0$ if and only the full Riemann curvature tensor of $\mathbf g$ vanishes at $p$. Define the set
 \begin{align*}
\begin{split}
     \mathbf{V}:=\big\{ p\in \Int \mathbf{N}\,\big|\, K(p)=0,\, \{\mathbf{Y}_0,\ldots, \mathbf{Y}_n\}\text{ is a Lorentzian orthonormal frame at }p,\\
     \text{ and } \nabla \mathbf Y_0(p)=\cdots =\nabla \mathbf Y_n(p)=0\big\}.
      \end{split}
 \end{align*} 
 We will prove that $\mathbf{V}=\Int \mathbf{N}$, which implies the desired result on all of $\mathbf{N}$ by continuity.
 
 Our first goal is to prove that $\mathbf{V}$ is non-empty, and here is where we must use our asymptotic assumptions.  The fact that $\mathbf{Y}_1$ is asymptotic to $\frac{\partial}{\partial x^1}$ guarantees that there exists a ball $B$ in the asymptotically flat end of $\mathbf{N}$ with the property that following any integral curve of $\mathbf{Y}_1$ from $B$ will lead out to spatial infinity in $\mathbf{N}$. Since $\mathbf{Y}_1$ is Killing, $K$ must be constant along an integral curve of $\mathbf{Y}_1$. But our asymptotic flatness assumptions tells us that $K(p)\to 0$ as $p$ approaches spatial infinity, so it follows that $K$ vanishes on $B$. Meanwhile, since $[\mathbf Y_1, \mathbf Y_\mu]=0$ and $\mathbf{Y}_1$ is Killing, it follows that for any $\mu, \nu$, the product $\mathbf{g}( \mathbf{Y}_{\mu},  \mathbf{Y}_{\nu})$ must be constant along integral curves of $\mathbf{Y}_1$. Since we know that the  $\mathbf{Y}_{\mu}$'s are asymptotically an orthonormal frame at infinity, it follows that they form an orthonormal frame at every point of~$B$. Finally, since the Riemann curvature vanishes on $B$, it is locally isometric to Minkowski space.
 Recall that we know exactly what every Killing field on Minkowski space looks like (even locally). Specifically, since the $\mathbf{Y}_{\mu}$'s have constant length
and the only Killing fields of Minkowski with constant length are the covariant constant ones (corresponding to translation symmetries), it follows that $\nabla\mathbf{Y}_{\mu}=0$ in $B$.  Hence $B\subset \mathbf{V}$.

It is clear that $\mathbf{V}$ is relatively closed in $\Int \mathbf{N}$. Therefore we can conclude that $\mathbf{V}=\Int \mathbf{N}$ if we can show that 
 $\mathbf{V}$ is also relatively open in $\Int\mathbf{N}$. Let $p\in \mathbf{V}$, and consider the set of all points that can be reached by following integral curves of elements of $\mathrm{span}(\mathbf{Y}_0,\ldots, \mathbf{Y}_n)$ originating at $p$. The key point is that since the  $\mathbf{Y}_\mu$'s are orthonormal at $p$, this set must include a small ball $B$ around $p$ in $\mathbf{N}$. But since $K(p)=0$ and $K$ must be constant along all of these integral curves for Killing fields, it follows that $K$ vanishes identically on $B$. Once again, using our knowledge of the Killing fields of Minkowski space, we know that since $\mathbf Y_\mu(p)\ne0$ and $\nabla \mathbf Y_\mu(p)=0$, it follows that $\mathbf Y_\mu$ must be covariant constant on $B$. Consequently, the $\mathbf{Y}_\mu$'s are orthonormal at every point of $B$. Thus $B\subset \mathbf{V}$, which is what we wanted to prove.

 \end{proof}

 By taking $(\mathbf N, \mathbf g)$ to be a trivial product, we can easily obtain a corresponding theorem in Riemannian geometry.
\begin{cor}\label{cor:KillingRiemannian}
Let $(M, g)$ be a connected $C^2$ Riemannian manifold, possibly with boundary, that is ``weakly\footnote{In other words, this result uses a much weaker definition of asymptotic flatness than the one given in Definition~\ref{def:AF}.}  asymptotically flat'' in the sense that a subset of $M$ is diffeomorphic to $\rr^n\setminus \bar B_1(0)$, and in these $x^1,\ldots,x^n$ coordinates, 
\[ g_{ij}=\delta_{ij} + O_2(|x|^{-\beta}),\]
for some $\beta>0$. Assume that $(M, g)$ is equipped with Killing fields $X_1,\ldots,X_n$ that are asymptotic to the coordinate translation directions in the sense that 
 \[ X_i=\frac{\partial}{\partial x^i} + O_2(|x|^{-\beta}). \]
Then $X_1,\ldots,X_n$ form a covariant constant global orthonormal frame, and in particular, $g$ must be flat.
\end{cor}
Again, note that there is no completeness assumption here, but when $(M, g)$ is complete (without boundary), it  follows from the Killing-Hopf Theorem (and a simple covering space argument as in the proof of Theorem~\ref{thm:main} below) that $(M, g)$ is isometric to Euclidean space.

In order to connect Lemma~\ref{lem:KID} to Theorem~\ref{thm:Killing0}, we will use the concept of a Killing development, first introduced by Beig and Chru\'{s}ciel~\cite{Beig-Chrusciel:1996}.
\begin{defin}
Let $(M, g)$ be a Riemannian manifold, and let $(f_0, X_0)$ be a lapse-shift pair on $M$ with $f_0>0$ on $M$. The \emph{Killing development of $(M, g)$ with respect to  $(f_0, X_0)$} is the Lorentzian manifold $(\mathbf N:=\rr\times M, \mathbf g)$ where
\begin{equation}\label{eq:KillingDev}
\mathbf g :=  -4f_0^2 (dx^0)^2 + g_{ij}(dx^i + X_0^i dx^0)(dx^j + X_0^j dx^0),
\end{equation}
where $x^0$ denotes a new coordinate on the $\rr$ factor and $x^i$ are just any coordinates on the $M$ factor.    
\end{defin}
Observe that $\mathbf{Y}_0=\frac{\partial}{\partial x^0}$ is a Killing field on $(\mathbf N, \mathbf g)$. 

Now suppose that $(M, g, \pi)$ is a vacuum initial data set equipped with a vacuum KID $(f_0, X_0)$ with $f_0>0$ (recall  Definition~\ref{def:KID}). In this case, one can verify that the Killing development $(\mathbf N, \mathbf g)$ is vacuum and contains the initial data set $(M, g, \pi)$ in the sense that the $x^0=0$ slice is just the Riemannian manifold $(M, g)$ with second fundamental form $k$ (corresponding to $\pi$ via~\eqref{eqn:k}). For details, see~\cite{Beig-Chrusciel:1996} or~\cite[Appendix B]{Huang-Lee}.
 Furthermore, by~\eqref{eq:Killing-lapse-shift-constraints}, one can see that every Killing field $\mathbf Y$ on $(\mathbf N, \mathbf g)$  restricts to a vacuum KID $(f, X)$ on $(M, g,\pi)$.
It turns out that this provides a one-to-one correspondence between Killing fields on $(\mathbf N, \mathbf g)$ and  vacuum KIDs on $(M, g,\pi)$. This fact is a simple special case of a theorem of Beig and Chru\'{s}ciel~\cite{Beig-Chrusciel:1997} that describes more general non-vacuum circumstances under which this one-to-one correspondence holds. A similar correspondence also holds for the maximal vacuum development (rather than the Killing development) by a classical result of  V. Moncrief~\cite{Moncrief:1975}.

\begin{proof}[Proof of Theorem~\ref{thm:main}]

Assume that $(M, g, \pi)$ satisfies the hypotheses of Theorem~\ref{thm:main}, but suppose that $E=0$. We want to show that $(M, g, \pi)$ embeds into Minkowski space. Invoking Lemma~\ref{lem:KID}, we can construct vacuum KIDs  $(f_0, X_0),(f_1,X_1),\ldots, (f_n,X_n)$ on $(M, g, \pi)$ such that on the specified end where $E=0$
\begin{align}
(2f_0, X_0)&=(1,0)+O_2(|x|^{-q}) \label{eq:f0} \\ 
(2f_i, X_i)&=\left(0,\frac{\partial}{\partial x^i} \right) +O_2(|x|^{-q}),\label{eq:fi}
\end{align}
for $i=1,\ldots,n$, where $q$ is the asymptotic decay rate of $(M, g, \pi)$.

Consider the open subset of $M$ where $f_0>0$, and let $U$ be the connected component of that set that contains our specified end. We define $(\mathbf N, \mathbf g)$ to be the Killing development of $(U, g)$ with respect to $(f_0, X_0)$. 
Note that asymptotic flatness of $(g,\pi)$ together with~\eqref{eq:f0} implies that the metric $\mathbf{g}$ given by~\eqref{eq:KillingDev} is asymptotically flat in the sense of Theorem~\ref{thm:Killing0}.

As discussed above, $(\mathbf N, \mathbf g)$ is vacuum and contains $(U, g, \pi)$ as the $x^0=0$ slice. Moreover, 
 each $(f_\mu, X_\mu)$ gives rise to a Killing field $\mathbf Y_\mu$. Note that the same argument used in the proof of Theorem~\ref{thm:Killing0} proves that for any $\mu, \nu =0, \ldots, n$, we have $[\mathbf Y_{\mu},\mathbf Y_{\nu} ] =0$. (This is because that part of the proof only used asymptotics of $\mathbf Y_\mu$ along $\{0\}\times U$, which we know because of~\eqref{eq:f0} and~\eqref{eq:fi}.)
 Recall that $\mathbf Y_0 =\frac{\partial}{\partial x^0}$. Since we also have $\mathbf Y_\mu = 2f_\mu \mathbf n + X_\mu$ along $\{0\}\times U$, we have  $\mathbf{n}=\frac{1}{2f_0} (\frac{\partial}{\partial x^0} -X_0)$, and consequently, 
 we also have
\[ 
\mathbf Y_i = \frac{f_i}{f_0}\left( \frac{\partial}{\partial x^0} -X_0\right)+X_i,
\]
for $i=1,\ldots,n$. Since $[\mathbf Y_0, \mathbf Y_i]=0$, this formula for $\mathbf Y_i$  holds on all $\mathbb R\times U$ rather than merely along $\{0\}\times U$, where $f_0, X_0, f_i, X_i$ are all extended to $\rr\times U$ to be independent of the $x^0$ variable. In particular, it follows that  $\mathbf Y_\mu$ is asymptotic to $\frac{\partial}{\partial x^\mu}$ in the sense required for application of Theorem~\ref{thm:Killing0}.

We can now invoke Theorem~\ref{thm:Killing0} to see that $\mathbf g$ is flat and $\{ \mathbf Y_0,\ldots, \mathbf Y_n\}$ is a covariant constant global orthonormal frame. In particular, $\mathbf{Y}_0$ has constant length $-1$ everywhere in~$\mathbf N$. Or in other words, $-4f_0^2 + |X_0|^2 = -1$ on all of $U$. This means that $2f_0\ge1$ on $U$, and consequently, $U$ must be all of $M$. So $(M, g, \pi)$ lies inside $(\mathbf N, \mathbf g)$.

The last step is to prove that $(\mathbf N, \mathbf g)$ is globally isometric to  Minkowski space. We will first consider the case where $M$ has no boundary, and then we will explain why $M$ having a nontrivial weakly outer trapped boundary is impossible.

Suppose $M$ has no boundary. Let $(\tilde{M}, \tilde{g}, \tilde{\pi})$ be the universal cover of $(M, g,\pi)$, and let 
 $(\tilde{\mathbf{N}}, \tilde{\mathbf{g}})$ be its Killing development with respect to the lift of $(f_0, X_0)$.
So  $\tilde{\mathbf{N}}$ is simply connected, and by what we proved above, $\tilde{\mathbf g}$ is flat. Moreover, the lack of boundary implies that
$(\tilde{\mathbf{N}}, \tilde{\mathbf{g}})$ is geodesically complete. (This is a nontrivial fact since $\mathbf g$ is Lorentzian, but it was previously explained in the proof of~\cite[Theorem 4.1]{Beig-Chrusciel:1996}.)
By the Killing--Hopf Theorem,\footnote{Technically, this is a Lorentzian version of the Killing--Hopf Theorem for flat manifolds, but the proof is  the same.} $(\tilde{\mathbf{N}}, \tilde{\mathbf{g}})$ must be isometric to Minkowski space. In particular, it follows that  $(\tilde{M}, \tilde{g}, \tilde{\pi})$ has only \emph{one} asymptotically flat end, and we claim that this implies that the covering map $\Pi:\tilde{M}\longrightarrow M$ is trivial, and the result follows. To see why the claim is true, let $\mathcal E$ be an asymptotically flat end of $M$, and consider the covering map from $\Pi^{-1}(\mathcal E)$ to $\mathcal E$ obtained by restriction of $\Pi$. Since $\mathcal E$ is simply connected, $\Pi^{-1}(\mathcal E)$ must be a disjoint union of isometric copies of $\mathcal E$. Since $\tilde{M}$ has only one asymptotically flat end, this implies that $\Pi$ is one-to-one.

 We show that $M$ cannot have a nontrivial weakly outer trapped boundary. Of course, it is well-known that it is impossible to have a weakly outer trapped boundary in Minkowski space, but we cannot apply this fact directly because at the moment we only know that our spacetime is both locally exactly Minkowski and asymptotically flat in the sense of Theorem~\ref{thm:Killing0}.  To complete the proof, we apply Theorem~\ref{thm:no-MOTS-stationary} where  we take ``$\Omega$'' to be the part of $M$ bounded by $\partial M$ and a large enough coordinate sphere so that it is strictly outer untrapped.

\end{proof}

\begin{rmk}
In the proof of Theorem~\ref{thm:main}, we primarily used the fact that $(g,\pi)$ is vacuum in order to  conclude that each $(f_i, X_i)$ extends to a Killing field on the Killing development of $(U,g)$ with respect to $(f_0, X_0)$. However, as mentioned earlier, a fascinating paper by Beig and Chru\'{s}ciel~\cite{Beig-Chrusciel:1997} explores when this is possible in non-vacuum scenarios, and it turns out that even without knowing that $(g,\pi)$ is vacuum, simply knowing that $(f_\mu, X_\mu)$ solves~\eqref{eqn:pair} turns out to be sufficient to conclude that each $(f_i, X_i)$ extends to a Killing field on the Killing development of $(U, g)$ with respect to $(f_0, X_0)$, and also that this Killing development satisfies the dominant energy condition~\cite{Huang-Lee}. This argument provides an alternative to invoking Corollary 6.6 of~\cite{Huang-Lee} when proving Theorem~\ref{thm:main}, but the argument that we presented above is simpler.
\end{rmk}

\section{Trapped and untrapped boundaries}\label{sec:boundary}

We prove Theorem~\ref{thm:no-MOTS-stationary} in this section, which is a direct consequence of the following more general result.  
\begin{thm}\label{thm:no-MOTS-alt}
Let $\Omega$ be a compact connected manifold with boundary such that $\partial \Omega=\Sigma_{\mathrm{in}}\sqcup \Sigma_{\mathrm{out}}$ where $\Sigma_{\mathrm{in}}$ and $\Sigma_{\mathrm{out}}$ are both closed and nonempty, and let $g_0$ be a $C^2$ Riemannian metric on $\Omega$ (which will only serve as a fixed Riemannian background metric). 
Define $\mathbf{N}= [0,\infty)\times \Omega$ and $h_0= dt^2 + g_0$, where $t$ denotes the coordinate function for the $[0,\infty)$ factor.

Let $\mathbf{g}$ be a Lorentzian metric on $\mathbf{N}$ that is $C^2$ in the sense that it extends to a $C^2$ Lorentzian metric on some manifold $\mathbf{N}'$ containing $\mathbf{N}$, and assume that $|\mathbf{g}|_{h_0}$ is uniformly bounded on~$\mathbf{N}$. We also assume that $t$ is a time function such that $\mathbf g \left(\frac{\partial}{\partial t}, \frac{\partial}{\partial t}\right)\le -\epsilon<0$ on all of $\mathbf{N}$ for some uniform constant $\epsilon$,
 and that each $\{t\}\times \Omega$ slice is spacelike. 

Now suppose that $\mathbf g$ satisfies the null energy condition, and that for 
all $t>0$, $\{t\}\times \Sigma_{\mathrm{out}}$ is a strictly outer untrapped surface (that is, $\theta^+>0$) with respect to the ``outward'' future null normal pointing out of $\mathbf{N}$. Then $\{0\}\times \Sigma_{\mathrm{in}}$ cannot be a weakly outer trapped surface (that is, it cannot have $\theta^+\le0$ on all of $\{0\}\times \Sigma_{\mathrm{in}}$) with respect to the ``outward'' future null normal pointing toward the interior of $\mathbf N$.
\end{thm}

Our proof of Theorem~\ref{thm:no-MOTS-alt} is elementary and does not depend on dimension. It is a Lorentzian analog of the following well-known theorem in Riemannian geometry, which follows from a classical argument of  T.~Frankel~\cite{Frankel:1966}. We thank Jose Espinar for pointing us toward Frankel's argument.

\begin{thm}\label{thm:Frankel}
Let $(\Omega, g)$ be a compact connected $C^2$ Riemannian manifold with boundary such that $\partial \Omega=\Sigma_{\mathrm{in}}\sqcup \Sigma_{\mathrm{out}}$ where $\Sigma_{\mathrm{in}}$ and $\Sigma_{\mathrm{out}}$ are both closed and nonempty. Suppose that $g$ has nonnegative Ricci curvature, and that $\Sigma_{\mathrm{out}}$ satisfies $H>0$ everywhere with respect to the normal pointing out of $\Omega$. Then $\Sigma_{\mathrm{in}}$ cannot satisfy  $H\le 0$ with respect to the normal pointing into $\Omega$. 
\end{thm}
 R. Ichida~\cite{Ichida:1981} proved the stronger statement that if one instead assumes that $H_{\Sigma_{\mathrm{out}}}\ge0$ and $H_{\Sigma_{\mathrm{in}}}\le0$, then $(\Omega, g)$ splits as a product.

\begin{proof}[Proof of Theorem~\ref{thm:no-MOTS-alt}]
By assumption, there is a larger spacetime manifold without boundary, called $\mathbf{N}'$, which contains $\mathbf{N}$, 
and all of our causal theoretic definitions below take place within the spacetime $\mathbf{N}'$, though it will be apparent that the geometry of $\mathbf{N}'$ outside $\mathbf{N}$ is irrelevant in the proof. We also extend the Riemannian background metric $h_0$ on $\mathbf{N}$ to a \emph{complete} Riemannian metric on $\mathbf{N}'$.

Our first claim is that the timelike future of $\{0\}\times \Sigma_{\mathrm{in}}$, $I^+\big( \{0\}\times \Sigma_{\mathrm{in}}\big) $, intersects $(0,\infty)\times \Sigma_{\mathrm{out}}$. Let $\sigma$ be any curve with unit speed with respect to $g_0$ that connects $\Sigma_{\mathrm{in}}$ to $\Sigma_{\mathrm{out}}$ in $\Omega$. Let $c>0$ be a constant, and consider the curve $\tilde \sigma(s):=(cs, \sigma(s))$ in $\mathbf{N}$ connecting $\{0\}\times \Sigma_{\mathrm{in}}$ to $(0,\infty)\times \Sigma_{\mathrm{out}}$. Then 
\begin{equation}\label{eqn:timelike-curve}
\mathbf g(\tilde \sigma'(s),\tilde \sigma'(s)) = c^2  \mathbf g \left(\tfrac{\partial}{\partial t}, \tfrac{\partial}{\partial t}\right)
+ 2c \mathbf g \left(\tfrac{\partial}{\partial t}, \sigma'(s) \right) + \mathbf g( \sigma'(s), \sigma'(s)) \le  -c^2\epsilon +2c|\mathbf g|_{h_0} + |\mathbf g|_{h_0},
\end{equation}
where $\epsilon>0$ is the constant assumed in the hypotheses, and $\sigma'(s)$ is interpreted to lie in $T_{\tilde\sigma(s)}\mathbf N$.
Since $|\mathbf g|_{h_0}$ is uniformly bounded, the right side of~\eqref{eqn:timelike-curve} will be negative for sufficiently large~$c$. Thus $\tilde{\sigma}$ is future-timelike for sufficiently large~$c$, proving the claim.

Let $T$ be the infimum over all $t$ such that $\{t\}\times \Sigma_{\mathrm{out}}$ intersects $I^+\big( \{0\}\times \Sigma_{\mathrm{in}}\big)$. So there is some point $p\in \Sigma_{\mathrm{out}}$ such that $(T,p)\in \partial I^+ \big( \{0\}\times \Sigma_{\mathrm{in}}\big)$. In fact, $(T,p)$ can be written as the limit of ending points of future timelike curves that begin in  $\{0\}\times \Sigma_{\mathrm{in}}$ and stay inside $[0,T+1]\times \Omega$. (If the timelike curve crosses through the $[0,\infty)\times \Sigma_{\mathrm{in}}$ part of the boundary, then the first part of the curve up to the last time it intersects $[0,\infty)\times \Sigma_{\mathrm{in}}$ can be replaced by an integral curve of $\frac{\partial}{\partial t}$, and if the timelike curve crosses through the $[0,\infty)\times \Sigma_{\mathrm{out}}$ part, then the last part of the timelike curve after the first time it intersects $[0,\infty)\times \Sigma_{\mathrm{out}}$ can be discarded.)
By a standard limit curve argument, there exists a limit curve $\gamma$ of (the time-reversed version of) the above-described family of timelike curves such that 
$\gamma$ is a past null geodesic starting at $(T,p)$ and lying in $\partial I^+ ( \{0\}\times \Sigma_{\mathrm{in}})$ such that $\gamma$ either ends at $ \{0\}\times \Sigma_{\mathrm{in}}$ or is inextendible. See, for example, \cite[Theorem 2.56]{Minguzzi:2019} or \cite[Proposition 3.4]{Galloway:vienna}.
Moreover, since $\gamma$ is a limit curve of curves lying in the compact set $[0,T+1]\times \Omega$, $\gamma$ itself must lie in $[0, T+1]\times \Omega$ as well. In fact, 
$\gamma$ lies in $[0,T]\times \Omega$ because it begins at $(T, p)$ and is past null. 

We claim that $\gamma$ is not inextendible, and consequently, it ends at $ \{0\}\times \Sigma_{\mathrm{in}}$. To prove the claim, recall that inextendible curves must have infinite $h_0$-length, where $h_0$ is the complete Riemannian background metric on the spacetime $\mathbf N'$.  See, e.g. \cite[Lemma 2.17]{Minguzzi:2019}. Note that since each $\{t\}\times \Omega$ is assumed to be spacelike, in the compact space $[0, T]\times \Omega$ there must be a uniform constant $\delta>0$ with the property that $\mathbf{g}(v,v) \ge \delta g_0(v,v)$ for all tangent vectors $v\in T\mathbf N$ that are tangent to the $\Omega$ directions in the splitting $[0,\infty)\times \Omega$.
Writing $\gamma(s) = ( t(s), \sigma(s))$, the null condition says that
\begin{multline}
0 = \mathbf g(\gamma'(s),\gamma'(s)) = \left(\tfrac{dt}{ds}\right)^2  \mathbf g \left(\tfrac{\partial}{\partial t}, \tfrac{\partial}{\partial t}\right)
+ 2 \left(\tfrac{dt}{ds}\right) \mathbf g \left(\tfrac{\partial}{\partial t}, \sigma'(s) \right) + \mathbf g( \sigma'(s), \sigma'(s)) \\
\ge  - \left(\tfrac{dt}{ds}\right)^2 |\mathbf  g|_{h_0} - 2 \left|\tfrac{dt}{ds}\right| |\mathbf g|_{h_0} |\sigma'(s)|_{g_0} + \delta |\sigma'(s)|^2_{g_0}, 
\end{multline}
where $\sigma'(s)$ is interpreted to lie in $T_{\gamma(s)}\mathbf N$. Solving this quadratic inequality implies that $|\sigma'(s)|_{g_0}\le C\left|\frac{dt}{ds}\right| $, where $C$ depends on $\delta$ and the uniform bound on $|\mathbf g|_{h_0}$.  Therefore the $h_0$-length of $\gamma$ can be bounded in terms of 
$\int_\gamma \left|\frac{dt}{ds}\right|$, which is bounded by $T$ since $t$ is a time function, $\gamma$ is past pointing, and $\gamma$ lies in $[0,T]\times \Omega$. 
Hence $\gamma$ has finite $h_0$-length, establishing the claim. 

\begin{figure}
    \centering
    \includegraphics[width=0.7\textwidth]{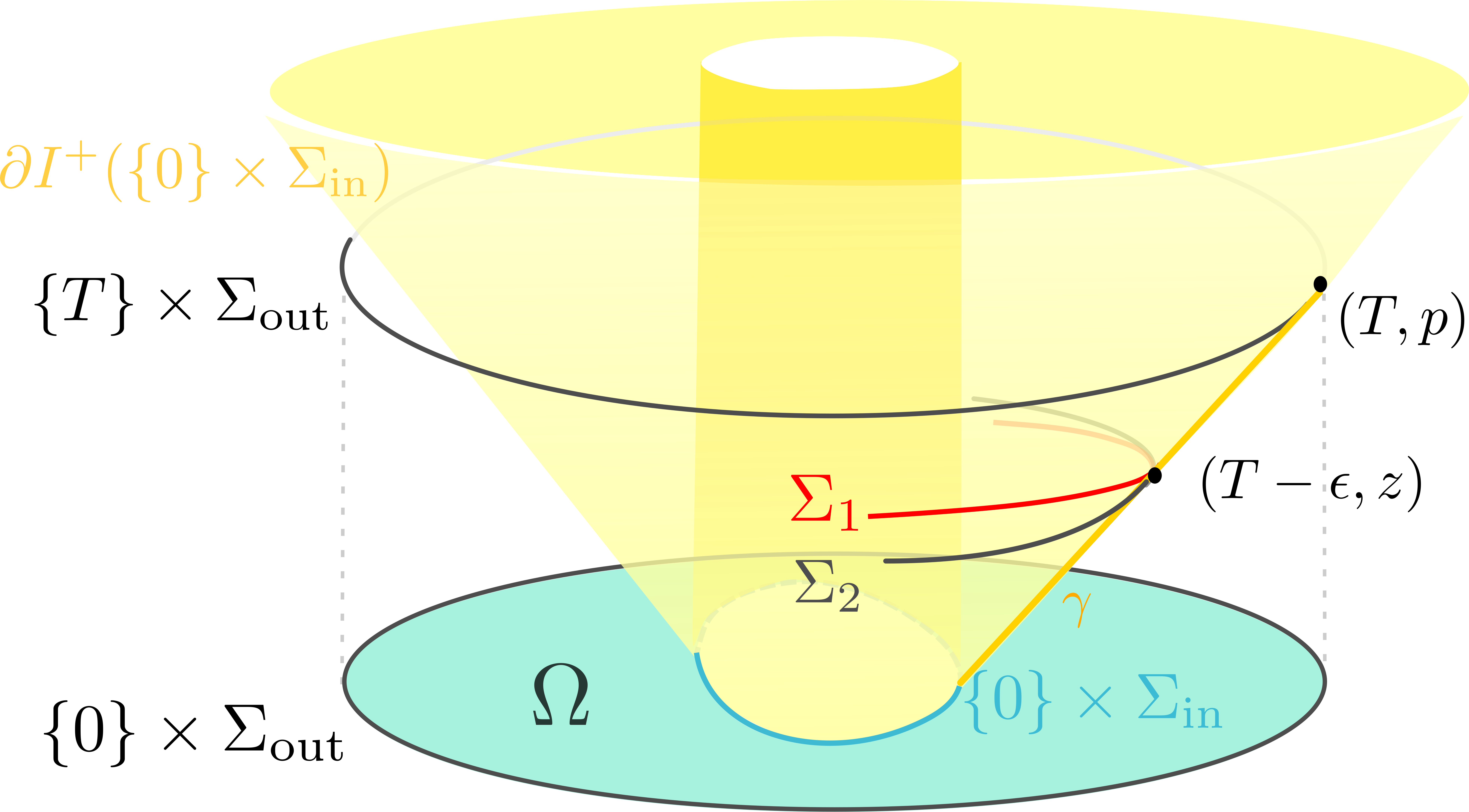}
    \caption{An illustration for the proof of Theorem~\ref{thm:no-MOTS-alt}. We first obtain a null geodesic $\gamma$ from $\{ 0 \}\times \Sigma_{\mathrm{in}}$ to $(T, p)\in \{ T \} \times \Sigma_{\mathrm{out}}$. Suppose on the contrary that $\theta_{\{ 0 \} \times \Sigma_{\mathrm{in}}}^+\le 0$. Then the Raychaudhuri equation and the null energy condition implies  $\theta_{\Sigma_1}^+\le 0$.  On the other hand, by construction,  $\theta^+_{\Sigma_2}>0$, and $\Sigma_2$ encloses  $\Sigma_1$ and is tangent to $\Sigma_1$ at $(T-\epsilon, z)$. This contradicts the strong comparison principle for $\theta^+$.
    } 
    \label{fig}
\end{figure}

With slight abuse of notation, let $\gamma:[0,a]\to \mathbf N$ denote the same geodesic described above, but with parameter reversed so that it is now a future null geodesic starting at $\gamma(0)\in \{0\}\times \Sigma_{\mathrm{in}}$ and ending at the point $\gamma(a)=(T, p)\in \{T\}\times \Sigma_{\mathrm{out}}$, 
lying entirely in $\partial I^+ ( \{0\}\times \Sigma_{\mathrm{in}})$.

 By a standard argument (see, e.g. \cite[Lemma 50, p. 298]{O'Neill:1983}), $\gamma$ must leave $\{0\}\times \Sigma_{\mathrm{in}}$ in an outward future null normal direction and not have any focal points in its interior (with respect to the surface $\{0\}\times \Sigma_{\mathrm{in}}$).  Moreover, the construction of $T$ guarantees that $\gamma$ is pointing in the outward future null normal direction when it hits $\{T\}\times \Sigma_{\mathrm{out}}$ at the point $(T, p)$ and that there are no focal points of $\{ T\}\times \Sigma_{\mathrm{out}}$ along $\gamma|_{(0, a]}$. (For otherwise, in either case there is a timelike curve from $\{ T\}\times \Sigma_{\mathrm{out}}$ to $\gamma(0)$, which contradicts the minimality of $T$.)

 It follows that the normal  exponential map on the normal bundle of $\{ 0 \} \times \Sigma_{\mathrm{in}}$ is nonsingular along $\gamma|_{[0, a)}$ \cite[Proposition 30, p. 283]{O'Neill:1983}, and thus $\gamma|_{[0, a)}$ is contained in a smooth null hypersurface, denoted by $\mathcal N_1$, generated by outward future  null normal geodesics starting from a small neighborhood of $\gamma(0)$ in $\{ 0 \}\times \Sigma_{\mathrm{in}}$. Similarly, there is a smooth null hypersurface, denoted by $\mathcal N_2$, generated by  past inward null normal geodesics starting from  a neighborhood of $\gamma(a)$ in $\{T\}\times \Sigma_{\mathrm{out}}$. Note that $\gamma$ lies in both $\mathcal N_1$ and $\mathcal N_2$.

Since the spacelike $t$-level sets must intersect the null hypersurfaces $\mathcal N_1$ and $\mathcal N_2$ transversely, the intersections are all smooth.
For sufficiently small $\epsilon>0$, let  $\Sigma_2$ be the smooth $t=T-\epsilon$
 slice of $\mathcal N_2$ with $\theta^+_{\Sigma_2}>0$ (because of our assumption that $\theta^+_{\{T\}\times \Sigma_{\mathrm{out}}}>0$), and let $\Sigma_1$ be the smooth  $t=T-\epsilon$ slice of $\mathcal N_1$.  By construction, $\gamma$ must intersect both $\Sigma_1$ and $\Sigma_2$ at some point $(T-\epsilon, z)$. Moreover, the minimality of $T$ guarantees that $\Sigma_2$ is ``outside'' of $\Sigma_1$ as hypersurfaces of $\{T-\epsilon\}\times \Omega$.

 Now assume that $\{0\}\times \Sigma_{\mathrm{in}}$ is weakly outer trapped, that is, $\theta^+_{\{0\}\times \Sigma_{\mathrm{in}}}\le0$. 
 Using a standard argument involving the Raychaudhuri equation and the null energy condition,  
 it follows that $\theta^+_{\Sigma_1}\le0$ at $(T-\epsilon, z)$. Since we already saw that $\theta^+_{\Sigma_2}>0$ at $(T-\epsilon, z)$ and $\Sigma_2$ encloses $\Sigma_1$, this contradicts the strong comparison principle for $\theta^+$ (see, for example, \cite[Proposition~3.1]{Ashtekar-Galloway:2005}).
\end{proof}

\begin{proof}[Proof of  Theorem~\ref{thm:no-MOTS-stationary}]
Let $\mathbf g$ be the Lorentzian metric and $\frac{\partial}{\partial t}$ be the global timelike Killing vector from the hypotheses of Theorem~\ref{thm:no-MOTS-stationary}. We verify that they satisfy the assumptions of Theorem~\ref{thm:no-MOTS-alt}. Let $h_0 := dt^2 + g_0$ where $g_0$ is the induced Riemannian metric of $\mathbf g$ on $\{ 0 \}\times \Omega$. Since $\frac{\partial}{\partial t}$ is a global timelike Killing vector field, the metric $\mathbf g$ is invariant in $t$, and thus $|\mathbf g|_{h_0}$ is uniformly bounded on $\mathbf N$ and the $t$-slices stay spacelike.  The compactness of $\Omega$ implies that   $\mathbf g(\tfrac{\partial}{\partial t}, \tfrac{\partial}{\partial t})\le -\epsilon <0$ for a uniform constant $\epsilon$ on $\{ 0\}\times \Omega$ and hence on all of $\mathbf N$.
\end{proof}


\bibliographystyle{alpha}
\bibliography{references}

\end{document}